\documentclass[reqno,12pt]{amsart}

\parskip = \bigskipamount
\usepackage{amssymb}
\usepackage{amscd}
\usepackage[all]{xy}
\usepackage{bbm}
\usepackage{mathrsfs}
\usepackage{enumerate}
\usepackage{tikz}
\usepackage[margin=1.5in]{geometry}
\usepackage{stmaryrd}

\newcommand\bcdot{\ensuremath{%
  \mathchoice%
   {\mskip\thinmuskip\lower0.2ex\hbox{\scalebox{1.5}{$\cdot$}}\mskip\thinmuskip}}%
   {\mskip\thinmuskip\lower0.2ex\hbox{\scalebox{1.5}{$\cdot$}}\mskip\thinmuskip}%
   {\lower0.3ex\hbox{\scalebox{1.2}{$\cdot$}}}%
   {\lower0.3ex\hbox{\scalebox{1.2}{$\cdot$}}}%
}

\def\XXint#1#2#3{{\setbox0=\hbox{$#1{#2#3}{%
\int}$ }
\vcenter{\hbox{$#2#3$ }}\kern-.6\wd0}}

\newcommand{\N}{\mathbb{N}}
\newcommand{\Z}{\mathbb{Z}}

\newcommand{\R}{\mathbb{R}}

\newcommand{\inv}{^{-1}}

\newcommand{\eps}{\varepsilon}

\newcommand{\tensor}{\otimes}

\newcommand{\Hom}{\operatorname{Hom}}
\newcommand{\del}{\nabla}
\newcommand{\ind}{\operatorname{ind}}

\newcommand{\lap}{\Delta}

\newcommand{\bd}{\partial}
\newcommand{\cl}{\overline}

\renewcommand{\div}{\operatorname{div}}

\newcommand{\RP}{{\R \text{\normalfont P}}}

\newcommand{\grad}{\del}
\newcommand{\vol}{\operatorname{vol}}

\newcommand{\f}{\colon}

\newtheorem{innercustomthm}{Theorem}
\newenvironment{customthm}[1]
  {\renewcommand\theinnercustomthm{#1}\innercustomthm}
  {\endinnercustomthm}

\theoremstyle{plain}
\newtheorem{theorem}{Theorem}[section]

\newtheorem{prop}[theorem]{Proposition}
\newtheorem{lem}[theorem]{Lemma}

\theoremstyle{definition}
\newtheorem{defn}[theorem]{Definition}
\newtheorem*{rem}{Remark}

\begin{document}

\title{A Weyl Law for the $p$-Laplacian}
\author{Liam Mazurowski}
\address{University of Chicago, Department of Mathematics, Chicago, IL 60637}
\email{maz@math.uchicago.edu}

\begin{abstract}
We show that a Weyl law holds for the variational spectrum of the $p$-Laplacian.  More precisely, let $(\lambda_i)_{i=1}^\infty$ be the variational spectrum of $\lap_p$ on a closed Riemannian manifold $(X,g)$ and let $N(\lambda) = \#\{i:\, \lambda_i < \lambda\}$ be the associated counting function. Then we have a Weyl law 
\[
N(\lambda) \sim c \vol(X) \lambda^{n/p}.
\]
This confirms a conjecture of Friedlander.  The proof is based on ideas of Gromov \cite{G} and Liokumovich, Marques, Neves \cite{LMN}.
\end{abstract}

\maketitle

\section{Introduction}

The classical Weyl law states that the Dirichlet eigenvalues $\lambda_i$ of $\lap$  on a domain $U\subset \R^n$ grow according to the asymptotics
\[
\#\{i:\, \lambda_i < \lambda\} \sim c \vol(U) \lambda^{n/2}
\]
where $c$ is a universal constant that depends only on $n$.  
Here the notation $f(\lambda) \sim g(\lambda)$ means that $f(\lambda)/g(\lambda) \to 1$ as $\lambda \to \infty$.  

Briefly, the idea of the proof is to use the variational characterization 
\begin{equation}
\label{vc}
\lambda_i = \inf_{i\text{-planes } P\subset W^{1,2}_0(U)}\bigg( \sup_{u\in P\setminus \{0\}} \frac{\int_U \vert \grad u\vert^2}{\int_U \vert u\vert^2}\bigg)
\end{equation}
to relate the eigenvalues of $\lap$ on $U$ with the eigenvalues of $\lap$ on a union of cubes that closely approximates $U$.  The eigenvalues of $\lap$ on a cube can be computed explicitly and the formula then follows.  The Weyl law also holds for the spectrum of $\lap$ on a closed Riemannian manifold.  This can be proved by studying the asymptotics of the heat kernel.

In \cite{AA}, the authors used variational methods to produce a sequence of eigenvalues $\lambda_i$ for the $p$-Laplacian $\lap_p$.  These eigenvalues are given by a min-max formula roughly similar to (\ref{vc}). 
 In \cite{F}, Friedlander studied the asymptotic growth of these eigenvalues of  $\lap_p$ and proved growth bounds of the form 
\[
C_1\vol(U)\lambda^{n/p} \le \#\{i:\, \lambda_i < \lambda\} \le C_2\vol(U)\lambda^{n/p}
\]
for some constants $C_1$ and $C_2$ that depend only on $n$ and $p$. 
Moreover, Friedlander conjectured that a Weyl law should hold in this setting.
In this paper, we prove the following theorem which confirms Friedlander's conjecture.

\begin{customthm}{A}
Let $1 < p < \infty$.  Let $(X^n,g)$ be a closed Riemannian manifold and let $(\lambda_i)$ be the variational spectrum of $\lap_p$ on $X$.  Then
\[
\#\{i:\, \lambda_i < \lambda\} \sim c \vol(X) \lambda^{n/p}
\]
where $c$ is a universal constant that depends only on $n$ and $p$.
\end{customthm}

\noindent The proof is based on a general framework for studying Weyl laws proposed by Gromov in \cite{G}.  We also use ideas from the proof of the Weyl law for the volume spectrum due to Liokumovich, Marques, and Neves \cite{LMN}.  It is worth noting that the proof seems to be new even in the case $p=2$ in the sense that it avoids the use of the heat kernel. 

\noindent {\bf Acknowledgements.}  I would like to thank my advisor Andr\' e Neves for suggesting this problem and for many helpful discussions. 

\section{Preliminaries: The Variational Spectrum of $\lap_p$}

Let $U \subset \R^n$ be a bounded open set with Lipschitz boundary.  If $u$ is a smooth function on $U$ then the $p$-Laplacian of $u$ is defined by 
\[
\lap_p u := \div(\vert \grad u\vert^{p-2}\grad u).
\]
The Dirichlet eigenvalue problem for $\lap_p$ asks for a function $u$ satisfying
\begin{equation}
\label{d}
\begin{cases}
-\lap_p u = \lambda \vert u\vert^{p-2} u, & \text{in } U\\
u = 0, & \text{on } \bd U.
\end{cases}
\end{equation}
A function $u\in W^{1,p}_0(U)$ is called a weak solution of (\ref{d}) provided 
\[
\int_U \vert \grad u\vert^{p-2} \grad u \cdot \grad \phi = \lambda \int_U \vert u\vert^{p-2} u \phi
\]
for every test function $\phi\in C^\infty_c(U)$.  If $u$ is a non-trivial weak solution of (\ref d) then $\lambda$ is called a Dirichlet eigenvalue for $\lap_p$ on $U$ and $u$ is called an eigenfunction with eigenvalue $\lambda$.  

It is possible to use variational methods to produce a sequence of eigenvalues for $\lap_p$ on $U$.  This was done in \cite{AA}.  There is also a detailed treatment in the book \cite{PAO}.  For clarity we outline the argument from \cite{PAO} below.  
Let $\mathcal M^0 = \mathcal M^0(U)$ be the set 
\[
\mathcal M^0 := \left\{u\in W^{1,p}_0(U):\, \| \grad u\|_{L^p(U)} = 1\right\}
\]
equipped with the topology it inherits as a subspace of $W^{1,p}_0$ with the norm topology.  
This is a $C^1$ Banach manifold (see \cite{PAO}).  There is an energy functional $E:\mathcal M^0\to \R$ given by 
\[
E(u) := \frac{\int_U \vert \grad u\vert^p}{\int_U \vert u\vert^p}
\]
and the eigenfunctions of (\ref{d}) are precisely the critical points of $E$ on $\mathcal M^0$. Moreover, $E$ satisfies the Palais-Smale compactness condition (\cite{PAO} Lemma 4.5). 

Critical points of $E$ can be produced using a min-max argument with the cohomological index.   The following discussion of the cohomological index is based on Chapter 2 of \cite{PAO}.  Note that there is a natural $\Z_2$-action on $\mathcal M^0$ and that $E$ respects this action, i.e., $E(u) = E(-u)$. 

\begin{defn}
A subset $A\subset \mathcal M^0$ is symmetric provided that $u\in A$ if and only if $-u\in A$ for every $u\in \mathcal M^0$.  
\end{defn}

\begin{defn}
A $\Z_2$-space is a Hausdorff, paracompact topological space equipped with a free $\Z_2$ action.
\end{defn}

\begin{defn}
Let $A$ and $B$ be $\Z_2$-spaces.  A map $f:A\to B$ is called odd provided 
$
f(-a) = -f(a) \text{ for all } a\in A.
$
\end{defn}

A symmetric set $A \subset \mathcal M^0$ is paracompact since $\mathcal M^0$ is a metric space and every metric space is paracompact.   Moreover, $A$ comes equipped with a free $\Z_2$-action $u\mapsto -u$. 
Thus $A$ is a $\Z_2$-space.  Now assume that $A$ is any $\Z_2$-space and let $\overline A$ be the quotient $A/\Z_2$. Then 
\[
\pi:A\to \overline A
\]
is a principal $\Z_2$-bundle over $\overline A$ and so there is a classifying map 
$
f\colon \overline A \to \RP^\infty.
$
Let $H^*_{AS}$ denote Alexander-Spanier cohomology.
The cohomology ring $H_{AS}^*(\RP^\infty;\Z_2)$ is isomorphic to $\Z_2[\sigma]$ where $\sigma$ is the non-zero element in $H_{AS}^1(\RP^\infty;\Z_2)$ and the classifying map induces a map in cohomology 
\[
f^*:H_{AS}^*(\RP^\infty;\Z_2) \to H_{AS}^*(\overline A;\Z_2).
\]
The following definition is originally due to Fadell and Rabinowitz \cite{FR}.
\begin{defn}
Let $A$ be a $\Z_2$-space and let $f$ be a classifying map for the bundle $\pi\f A\to \overline A$.  Then the cohomological index of $A$ is 
\[
\ind_{AS}(A) = \sup\{k\ge 1:\, f^*(\sigma^{k-1}) \neq 0 \text{ in } H_{AS}^*(\overline A;\Z_2)\}.  
\]
By convention  $\ind_{AS}(\emptyset) = 0$.
\end{defn}
\noindent
The properties of the cohomological index will be discussed further in the next section of the paper.  


It remains to perform the min-max argument.  For complete details see Chapter 4 of \cite{PAO}.  Let $\mathcal F^0 = \mathcal F^0(U)$ be the collection of all symmetric subsets of $\mathcal M^0$.  Define the classes
\[
\mathcal F_k^0 = \mathcal F_k^0(U) = \{A \in \mathcal F^0:\, \ind_{AS}(A) \ge k\}
\]
and then consider the min-max values 
\[
\lambda_k = \lambda_k(U) = \inf_{A\in \mathcal F_k^0} \sup_{u\in A} E(u).  
\]
By \cite{PAO} Theorem 4.6, the numbers $\lambda_k$ are Dirichlet eigenvalues of $\lap_p$ on $U$ and $\lambda_k \to \infty$ as $k\to\infty$.   Moreover, if we define the counting function 
\[
N_U^0(\lambda) = \#\{j:\, \lambda_j < \lambda\}
\]
then $N_U^0$ satisfies $N_U^0(\lambda) = \ind_{AS}(E\inv[0,\lambda))$.   The sequence $(\lambda_k)$ is called the variational spectrum of $\lap_p$ on $U$.  
It doesn't seem to be known (see \cite{PAO} page 71) whether the variational spectrum contains every eigenvalue of $\lap_p$.

\begin{rem}
Let $A$ be a $\Z_2$-space with classifying map $f$.  Let $H^*$ denote singular cohomology and let $\sigma$ be the non-zero element in $H^1(\RP^\infty;\Z_2)$.  Define 
\[
\ind(A) = \sup\{k\ge 1:\, f^*(\sigma^{k-1})\neq 0 \text{ in } H^*(\overline A;\Z_2)\}.
\]
If $A$ is locally contractible then the Alexander-Spanier cohomology of $A$ is isomorphic to the singular cohomology of $A$ (see \cite{S}) and thus $\ind_{AS}(A) = \ind(A)$.  Suppose now that $A = E\inv[0,\lambda) \subset \mathcal M^0$. Note that $\mathcal M^0$ is locally contractible since it is a Banach manifold.  Thus $A$ is also locally contractible since it is an open subset of a locally contractible space.   It follows that $\ind_{AS}(A) = \ind(A)$.  In particular, the counting function satisfies $N^0_U(\lambda) = \ind(E\inv[0,\lambda))$.
\end{rem}

\section{Preliminaries: The Neumann Problem}

It will be useful to simultaneously investigate the corresponding Neumann eigenvalue problem
\[
\begin{cases}
-\lap_p u = \lambda \vert u\vert^{p-2}u, &\text{in } U\\
\frac{\bd u}{\bd \nu} = 0, &\text{on } \bd U.
\end{cases}
\]
To this end, define the set 
\[
\mathcal M = \mathcal M(U) := \left\{ u\in W^{1,p}(U):\, \| u\|_{W^{1,p}(U)} = 1\right\}.
\]
Again there is an energy $E:\mathcal M\to \R$ given by 
\[
E(u) = \frac{\int_U \vert \grad u\vert^p}{\int_U \vert u\vert^p}.
\]
Define the counting function
\[
N_U(\lambda) = \ind(E\inv[0,\lambda)).
\]
Then $N_U^0(\lambda) \le N_U(\lambda)$ for every $\lambda$. 
We will show that the Neumann counting function also satisfies a Weyl law $N_U(\lambda) \sim c \vol(U) \lambda^{n/p}$. 

\begin{rem}
Throughout the paper objects associated with the Dirichlet problem will be decorated with a superscript zero while the corresponding objects in the Neumann problem will appear without decoration.  Thus $N^0_U$ denotes the counting function for the Dirichlet problem while $N_U$ denotes the counting function for the Neumann problem, and so on.  This is consistent with the notation in \cite{G}.
\end{rem}

\section{Preliminaries: Properties of the Cohomological Index}

This section collects a few properties of the cohomological index that will be needed later.
Proofs of the following properties can be found in \cite{FR} or \cite{PAO}.  Since the proofs are relatively short, we reproduce them below for the convenience of the reader.
 
\begin{prop}
\label{st}
The index $\ind$ satisfies the following properties.
\begin{enumerate}
\item[(i)] If $A$, $B$ are $\Z_2$-spaces and there is an odd continuous map $f:A\to B$ then $\ind(A) \le \ind(B)$.
\item[(ii)] If $X$ is a $\Z_2$-space and $A,B\subset X$ are open, paracompact symmetric subsets with $X = A\cup B$ then $\ind(X) \le \ind(A) + \ind(B)$.
\end{enumerate}
\end{prop}

\begin{proof}
(i) Let $\overline f:\overline A \to \overline B$ be the induced map and let $g:\overline B\to \RP^\infty$ be a classifying map.  Then $g \overline f:\overline A \to \RP^\infty$ serves as a classifying map for $\overline A$.  Notice that there are maps
\[
H^*(\RP^\infty;\Z_2) \overset{g^*}{\longrightarrow} H^*(\overline B;\Z_2) \overset{\overline f^*}\longrightarrow H^*(\overline A;\Z_2)
\]
and thus $\ind(A)\le \ind(B)$.  

(ii) Let $f:\overline X\to \RP^\infty$ be a classifying map and let 
\[
i_A:\overline A\to \overline X, \quad i_B:\overline B\to\overline X
\]
be the inclusions.  Then $f \iota_A$ is a classifying map for $\overline A$ and $f\iota_B$ is a classifying map for $\overline B$.  Without loss, we may assume that $\ind(A) = i$ and $\ind(B) = j$ are finite.  Put $\theta = f^*\sigma$.  Then $\iota_A^*\theta^i = 0$ and $\iota_B^*\theta^j = 0$.  There are exact sequences
\begin{gather*}
H^i(\overline X,\overline A) \to H^i(\overline X) \to H^i(\overline A),\\
H^j(\overline X,\overline B) \to H^j(\overline X) \to H^i(\overline B),
\end{gather*}
and hence there are classes $\theta_A \in H^i(\overline X;\overline A)$ and $\theta_B\in H^j(\overline X;\overline B)$ that map to $\theta$. 

Since $A$, $B$ are open in $X$ there is a relative cup product 
\[
\smile:H^*(\overline X,\overline A) \times H^*(\overline X,\overline B) \to H^*(\overline X, \overline A\cup \overline B).
\]  
Moreover, since $\overline X = \overline A \cup \overline B$, it follows that $\theta_A \smile \theta_B = 0$ in $H^*(\overline X,\overline A \cup \overline B)$.  By naturality, $\theta^{i+j}$ is the image of $\theta_A \smile \theta_B$ under the map 
\[
H^*(\overline X,\overline A \cup \overline B) \to H^*(\overline X).
\]
Therefore $\theta^{i+j} = 0$ and the result follows.
\end{proof}

One further property will be required.  Assume that $A$ and $B$ are $\Z_2$-spaces.  By definition, their join is the quotient  
\[
A*B := (A\times B\times [0,1])/\sim
\]
where $(a_1,b,0)\sim (a_2,b,0)$ and $(a,b_1,1)\sim(a,b_2,1)$.  This is also $\Z_2$-space in a natural way.  

\begin{prop}
\label{pj}
Let $A$ and $B$ be $\Z_2$-spaces.  Then 
\[
\ind(A) + \ind(B) \le \ind(A* B).
\]  
\end{prop}

\begin{rem}
The special case of this proposition where $B = S^0$ is proven in \cite{FR}.  Presumably the general case is also known, but we could not find a reference in the literature and hence we provide a proof below.
\end{rem}

Our proof of Proposition \ref{pj} is based on a join operation in homology constructed in \cite{GKPS}.  
We summarize the construction below.  
Let $X$ and $Y$ be topological spaces.  
If $\Delta^m$ and $\Delta^n$ are the standard simplices, then there is a natural identification $\Delta^m * \Delta^n \cong \Delta^{m+n+1}$.  Thus given singular simplices $\alpha:\Delta^m \to X$ and $\beta:\Delta^n \to Y$  one can form a new singular simplex 
\[
\alpha * \beta: \Delta^{m+n+1} \to X*Y.
\]
In the following all groups have $\Z_2$-coefficients, even where this is not explicitly indicated in the notation.  
Extending linearly, there is a map on chains
\[
*: C_m(X) \tensor C_n(Y) \to C_{m+n+1}(X*Y)
\]
and since 
\[
\bd(\alpha * \beta) = (\bd\alpha) * \beta + \alpha * (\bd\beta),
\]
this descends to a map in homology 
\[
*:H_m(X)\tensor H_n(Y) \to H_{m+n+1}(X*Y). 
\]
Using this, it is possible to construct an equivariant join operation on $\Z_2$-spaces.  

Suppose now that $X$ and $Y$ are $\Z_2$-spaces and recall that $\overline X = X/\Z_2$ and $\overline Y = Y/\Z_2$.   
Let $g$ be the antipodal map on some $\Z_2$-space $Z$.  Then a chain $c$ in $Z$ is called $\Z_2$-equivariant provided $g_\# c = c$.  Following \cite{GKPS}, let $C_*(Z)^{\Z_2}$ denote the set of $\Z_2$-equivariant chains in $Z$.  There is a natural identification 
\[
C_*(\overline Z) \leftrightarrow C_*(Z)^{\Z_2}
\] 
given by sending $\alpha:\Delta^m \to \overline Z$ to the sum of its two lifts to $Z$.  Since the join operation respects equivariance, there is an induced operation  
\begin{gather*}
*: C_m(X)^{\Z_2} \tensor C_n(Y)^{\Z_2} \to C_{m+n+1}(X * Y)^{\Z_2}.
\end{gather*}
Using the above equivalence, this gives an operation 
\[
\star:C_m(\overline X) \tensor C_n(\overline Y) \to C_{m+n+1}(\overline {X *  Y})
\]
and again this descends to give an operation in homology
\[
\star:H_m(\overline X) \tensor H_n(\overline Y) \to H_{m+n+1}(\overline {X*  Y}).
\]
It is possible to compute this map in the case where $X= S^\infty$ and $Y=S^\infty$.  

Take $X = S^\infty$ and $Y = S^\infty$.     Note that $\overline X$ and $\overline Y$ are both homeomorphic to $\RP^\infty$ and hence $H_*(\overline X) \cong H_*(\overline Y) \cong \Z_2[x]$ in the sense of additive groups.  Now think of $S^\infty$ as the set of points $(x_i)_{i\in \N}$ in $\R^\infty$ such that 
\[
\sum_{i=1}^\infty x_i^2 = 1
\]
and all but finitely many $x_i$ are equal to 0.  There is a homeomorphism $j:X * Y \to S^\infty$ given by 
\[
j((x_i),(y_i),t) = (x_1\sqrt t , y_1\sqrt{1-t},x_2\sqrt t,y_2\sqrt{1-t},\hdots).
\]
This gives an identification of $X*Y$ with $S^\infty$ and hence $H_*(\overline {X*Y})$ is also isomorphic to $\Z_2[x]$.  

Fix a pair of non-negative integers $m$ and $n$.  Let $\mathcal X$ be the set of points $(x_i)$ in $X$ such that $x_i = 0$ for all $i> m+1$, and let $\mathcal Y$ be the set of points $(y_i)$ in $Y$ such that $y_i = 0$ for all $i > n+1$.  Let 
\[
\overline {\mathcal X} \subset  \overline X, \quad 
\overline {\mathcal Y} \subset \overline Y
\]
be the quotients.  Then $[\overline {\mathcal X}]$ is the non-trivial element of $H_m(\overline X)$ and $[\overline {\mathcal Y}]$ is the non-trivial element of $H_n(\overline Y)$.  Let $\mathcal Z$ be the set of points $(z_i)$ in $S^\infty$ such that $z_{2i-1} = 0$ for all $i > m+1$ and $z_{2j} = 0$ for all $j > n+1$.  
Then from the definition of the operation $\star$, it follows that 
\[
j_*([\overline{\mathcal X}]\star [\overline {\mathcal Y}]) = [\overline {\mathcal Z}].
\]
Since $j_*$ is an isomorphism, this means that $[\overline{\mathcal X}]\star [\overline {\mathcal Y}]$ is the non-trivial class in $H_{m+n+1}(\overline {X*  Y})$.


\begin{proof}(Proposition \ref{pj})
Let $\overline A = A/\Z_2$ and $\overline B = B/\Z_2$ and let 
\[
f:\overline A \to \RP^\infty, \quad g:\overline B\to \RP^\infty
\]
be classifying maps for the bundles $A\to\overline A$, $B\to \overline B$.  Recall that there is a homeomorphism $j:S^\infty * S^\infty \to S^\infty$.
Since $f,g$ are classifying maps, we get the following commutative diagrams.
\begin{center}
\begin{tikzpicture}
\node(1) at (0,0) {$A$};
\node(2) at (0,-1.5) {$\overline A$};
\node(3) at (2,0) {$S^\infty$};
\node(4) at (2,-1.5) {$\RP^\infty$};

\draw[->] (1)->(2);
\draw[->] (3)->(4);
\draw[->] (1)->node[above]{$\widetilde f$} (3);
\draw[->] (2)->node[below]{$f$}(4); 

\node(5) at (4,0) {$B$};
\node(6) at (4,-1.5) {$\overline B$};
\node(7) at (6,0) {$S^\infty$};
\node(8) at (6,-1.5) {$\RP^\infty$};

\draw[->] (5)->(6);
\draw[->] (7)->(8);
\draw[->] (5)->node[above]{$\widetilde g$} (7);
\draw[->] (6)->node[below]{$g$}(8); 
\end{tikzpicture}
\end{center}
Thus it is possible to define a map $\widetilde h = \widetilde f * \widetilde g:A*B\to S^\infty*S^\infty$.
Moreover $\widetilde h$ is odd and so it induces a map 
\[
h:\overline{A*B} \to \overline {S^\infty * S^\infty} \cong \RP^\infty.
\]
This is a classifying map for the bundle $A*B\to \overline {A*B}$.

Since $\Z_2$ is a field, the universal coefficient theorem implies that 
\[
H^i(X;\Z_2) \cong \Hom(H_i(X;\Z_2),\Z_2).
\]
Now assume $m+1 \le \ind(A)$ and $n+1\le \ind(B)$.  Then $f^*(\sigma^{m})$ and $g^*(\sigma^{n})$ are non-zero and hence there exist classes 
\[
\alpha\in H_{m}(\overline A;\Z_2), \quad \beta\in H_{n}(\overline B;\Z_2)
\]
such that $f_*\alpha$ and $g_*\beta$ are non-zero.  Define $\gamma = \alpha \star \beta \in H_{m+n+1}(\overline {A * B})$.  By naturality of the above construction, it follows that 
\[
h_*(\gamma) = (f_*\alpha)\star (g_*\beta), \quad \text{in } H_{m+n+1}(\overline{S^\infty *S^\infty}).
\]
But we know the class on the right hand side of the above equation is non-zero and therefore $h^*(\sigma^{m+n+1})$ is non-zero.  The result follows.
\end{proof}

\section{The Dirichlet Domain Monotonicity Inequality}
\label{dirin}

Assume that $V,W$ are disjoint open sets with $V,W\subseteq U$.  Recall that when $p=2$ the Dirichlet eigenvalues of the Laplacian satisfy a domain monotonicity inequality of the form
\[
N^0_U(\lambda) \ge N^0_{V}(\lambda) + N^0_W(\lambda). 
\]
We want to show that a similar inequality holds for arbitrary $p$.  The arguments in this section closely follow Gromov in \cite{G}.  The first step is to prove an inequality relating the energy of functions $v$ on $V$ and $w$ on $W$ with the energy of $v+w$ on $U$.

\begin{lem}[See \cite{G} Lemma 3.2.A]
\label{ei}
Assume $V,W$ are disjoint open sets with $V,W \subset U$.  Let $v\in W^{1,p}_0(V)$ and $w\in W^{1,p}_0(W)$ not both zero and define $u := v+w \in W^{1,p}_0(U)$.  Then $E(u) \le \max\{E(v),E(w)\}$.  
\end{lem} 

\begin{proof}
This is obvious if either $v \equiv 0$ or $w \equiv 0$, so we may assume that $v$ and $w$ are not identically 0.  Without loss of generality, we may also assume that $E(w) \le E(v)$.  Define the numbers 
\begin{align*}
a := \int_V \vert \grad v\vert^p, \quad b:= \int_W \vert \grad w\vert^p, \quad 
c := \int_V \vert v\vert^p, \quad d:= \int_W \vert w\vert^p.  
\end{align*}
Elementary manipulations show that 
\[
\frac {a+b}{c+d} \le \frac{a}{c} \iff \frac{b}{d} \le \frac a c,
\]
and the inequality on the right holds since $E(w) \le E(v)$.  It follows that 
\[
E(u) = \frac {a+b}{c+d} \le \frac{a}{c} = E(v).
\]
and the lemma is proven.
\end{proof}

The next step is to prove a monotonicity inequality.  

\begin{prop}[See \cite{G} 3.2.A$_1$]
\label{sa}
Assume $V,W$ are disjoint open sets with $V,W\subset U$.  Then we have $N_U^0(\lambda) \ge N^0_V(\lambda) + N^0_W(\lambda)$.  
\end{prop} 

\begin{proof}
Define the sets 
\begin{gather*}
A := \{v\in \mathcal M^0(V):\, E(v) < \lambda\},\\
B := \{w\in \mathcal M^0(W):\, E(w) < \lambda\},\\
C := \{u\in \mathcal M^0(U):\, E(u) < \lambda\}.
\end{gather*} 
Then $N^0_V(\lambda) = \ind(A)$, $N^0_W(\lambda) = \ind(B)$,  $N^0_U(\lambda) = \ind(C)$ and hence we must verify that 
\[
\ind(A) + \ind(B) \le \ind(C).
\]
Consider the join $A* B$.   The map $A*B \to \mathcal M^0(U)$ given by
\[
(v,w,t) \mapsto tv + (1-t)w 
\]
is a homeomorphism onto its image.  Hence $A*B$ can be viewed as a subset of $\mathcal M^0(U)$.    
Lemma \ref{ei} shows that actually $A* B \subset C$ and thus 
\[
\ind(A* B) \le \ind(C)
\]
by Proposition \ref{st}(i).  
But Proposition \ref{pj} says 
\[
\ind(A) + \ind(B) \le \ind(A*B).
\]
Combining these inequalities gives $\ind(A) + \ind(B) \le \ind(C)$, as needed.
\end{proof}

It is possible to be slightly more general.    Given $U$ and a positive real number $a$, let $aU$ be the set $\{ax:\, x\in U\}$.  Then the scaling properties of the energy lead to a relationship between $N^0_U(\lambda)$ and $N^0_{aU}(\lambda)$.  This is the content of the following proposition.  

\begin{prop}
\label{es}
Let $a$ be a positive real number.  Then $N^0_{aU}(\lambda) = N^0_U(a^p\lambda)$.   
\end{prop}

\begin{proof}
Define a map $g:W^{1,p}_0(U)\setminus\{0\} \to W^{1,p}_0(aU)\setminus\{0\}$ by 
\[
(gu)(x) = u(x/a).  
\]
This map induces a homeomorphism 
\[
g:\mathcal M^0(U) \cong \mathcal M^0(aU).
\]
Thus $\ind(A) = \ind(g(A))$ for every symmetric $A\subset \mathcal M^0(U)$.  Moreover, a straightforward calculation shows that 
\[
E(gu) = a^{-p}E(u)
\]
for every $u\in \mathcal M^0(U)$.  The result follows.
\end{proof}

Now suppose that $U,U_1,\hdots,U_m$ are open sets in $\R^n$ and that $a_1,\hdots,a_m$ are positive real numbers.  Following Gromov \cite{G}, we will write 
\[
\sum_{i=1}^m a_iU_i \prec U
\]
if there exist elements $b_1,\hdots,b_m\in \R^n$ such that the translates $a_iU_i+b_i$ are all disjoint and contained in $U$.  Using Proposition \ref{sa} and Proposition \ref{es} and induction shows that 
\begin{equation}
\label{ddm}
\sum_{i=1}^m a_iU_i \prec U \implies N^0_U(\lambda) \ge \sum_{i=1}^m N_{U_i}^0(a_i^p\lambda).
\end{equation}
We will refer to this as the Dirichlet domain monotonicity inequality.

\section{The Weyl Law for Dirichlet eigenvalues} \label{DWL}

In this section we prove the Dirichlet Weyl law for domains in $\R^n$. The first step is to prove the Weyl law for a cube.  The argument is essentially the same as the proof of Lemma 3.3 in \cite{LMN}.   Also see the Trivial Lemma in Section 3.4 of \cite{G}.

\begin{lem}
Let $C$ be the unit cube in $\R^n$ and define $f(\lambda) = \lambda^{-n/p} N^0_C(\lambda)$.  Then $f$ tends to a limit as $\lambda \to \infty$.    
\end{lem}

\begin{proof}
Choose sequences $(\lambda_j)$, $(\mu_k)$ so that 
\[
\limsup_{\lambda\to \infty} f(\lambda) = \lim_{j\to \infty} f(\lambda_j), \quad \liminf_{\lambda \to \infty} f(\lambda) = \lim_{k\to\infty} f(\mu_k).
\]
Now fix some $j$ and consider $k$ large.  Let $M_k$ be the largest integer such that it is possible to pack $M_k$ disjoint open cubes of volume $(\lambda_j/\mu_k)^{n/p}$ into $C$.  Then  
\[
\sum_{\ell=1}^{M_k} ({\lambda_j}/ {\mu_k})^{1/p} C \prec C
\]
and it follows by the domain monotonicity inequality (\ref{ddm}) that 
\[
N_C^0(\lambda) \ge M_k N_C^0(\lambda_j\lambda/\mu_k).
\]
Choose $\lambda = \mu_k$ and multiply both sides by $\mu_k^{-n/p}$ to get
\[
\mu_k^{-n/p} N_C^0(\mu_k) \ge M_k \mu_k^{-n/p} N_C^0(\lambda_j).
\]
Now let $k\to \infty$ and use the fact that $M_k\mu_k^{-n/p} \to \lambda_j^{-n/p}$ to get 
\[
\liminf_{\lambda\to \infty} f(\lambda) \ge \lambda_j^{-n/p} N_C^0(\lambda_j).
\]
Finally let $j\to \infty$ to get 
\[
\liminf_{\lambda \to \infty} f(\lambda) \ge \limsup_{\lambda \to \infty} f(\lambda).
\]
The result follows.
\end{proof}

Define $c^0 = \lim_{\lambda\to\infty} f(\lambda)$.  
As mentioned in Section 1, estimates of Friedlander \cite{F} imply that 
\[
C_1\lambda^{n/p} \le N_C^0(\lambda) \le C_2\lambda^{n/p}, \quad \text{as } \lambda\to\infty.
\]
It follows at once that $0 < c^0 < \infty$.  Thus the Weyl law holds on the unit cube, i.e., we have $N_C^0(\lambda) \sim c^0\lambda^{n/p}$. 

It is also possible to give a self-contained proof that $0 < c^0 < \infty$ which does not rely on the estimates of Friedlander.  The next proposition shows that $c^0 > 0$.  We will give the proof that $c^0 < \infty$ in a later section.

\begin{prop}
Put $c^0 = \lim_{\lambda\to\infty} \lambda^{-n/p}N^0_C(\lambda)$.  Then $c^0 > 0$.
\end{prop}

\begin{proof}
As above let $f(\lambda) = \lambda^{-n/p}N^0_C(\lambda)$.  Fix some $\lambda > 0$ and some positive integer $k$.    It is possible to divide unit cube into $k^n$ disjoint open cubes of volume $k^{-n}$.  Hence  
\[
\sum_{\ell=1}^{k^n} k^{-1}C \prec C
\]
and so the domain monotonicity inequality implies 
\[
N^0_C(\lambda) \ge k^n N^0_C(k^{-p} \lambda).
\]
Multiplying by $\lambda^{-n/p}$ yields
\[
f(\lambda) = \lambda^{-n/p} N^0_C(\lambda) \ge (k^{-p}\lambda)^{-n/p} N^0_C(k^{-p}\lambda) = f(k^{-p}\lambda).
\]
Now there exists some $\lambda_1$ such that $N^0_C(\lambda) \ge 1$ for all $\lambda \ge \lambda_1$.  Hence $f$ has a positive minimum $C_1$ on the interval $[\lambda_1, 2^p\lambda_1]$.  Since 
\[
(k+1)^p \le (2k)^p, \quad \text{for all } k\ge 1
\]
it follows that 
\[
[\lambda_1, \infty) = \bigcup_{k=1}^\infty [k^p\lambda_1,k^p2^p\lambda_1].
\]
Therefore, given an arbitrary $\lambda\ge \lambda_1$, there exists some positive integer $k$ such that $k^{-p}\lambda \in [\lambda_1,2^p\lambda_1]$.  Using the above inequality, it follows that 
\[
f(\lambda) \ge f(k^{-p}\lambda) \ge C_1.
\]
This proves that $c^0 \ge C_1 > 0$, as needed. 
\end{proof}

Finally we derive the Weyl law for a general domain in $\R^n$. 

\begin{theorem}
Let $U\subset \R^n$ be a bounded open set with Lipschitz boundary.  Then  $N_U^0(\lambda) \sim c^0\vol(U)\lambda^{n/p}$ as $\lambda\to\infty$.
\end{theorem} 

\begin{proof}
Without loss of generality we can assume that $U$ has unit volume.  Define the numbers 
\[
\underline \beta(U) = \liminf_{\lambda\to \infty} \lambda^{-n/p} N_U^0(\lambda),\quad \overline\beta(U) = \limsup_{\lambda \to \infty} \lambda^{-n/p} N_U^0(\lambda).
\]
Let $\eps > 0$ and choose numbers $a_1,\hdots,a_m$ such that 
\[
\sum_{i=1}^m a_i C \prec U
\]
and $\sum_{i=1}^m a_i^n \ge 1 -\eps$.  Applying the domain monotonicity inequality gives 
\[
N_U^0(\lambda) \ge \sum_{i=1}^m N_C^0(a_i^p\lambda).
\]
But $N_C^0(a_i^p\lambda) \sim c^0a_i^{n}\lambda^{n/p}$ as $\lambda \to \infty$ since the Weyl law holds for cubes.  Hence multiplying both sides of the above inequality by $\lambda^{-n/p}$ and letting $\lambda\to \infty$ it follows that 
\[
\underline \beta(U) \ge c^0 \left( \sum_{i=1}^m a_i^n\right) \ge c(1-\eps).
\]
Since $\eps$ was arbitrary this implies that $\underline \beta(U) \ge c^0$.  

It remains to show that $\overline \beta(U) \le c^0$.  To this end, fix $a > 0$ so that some translate of $aU$ is contained in $C$.  Let $\eps > 0$ and choose numbers $a_1,\hdots,a_m$ so that 
\[
aU + \sum_{i=1}^m a_i C \prec C 
\]
and $\sum_{i=1}^m a_i^n \ge 1 - a^n - \eps$.  Then  
\[
N_C^0(\lambda) \ge N_U^0(a^p\lambda) + \sum_{i=1}^m N_C^0(a_i^p\lambda).
\]
Pick a sequence $\lambda_j \to \infty$ so that 
\[
a^{-n}\lambda_j^{-n/p} N_U^0(a^p\lambda_j) \to \overline \beta(U).
\]
Taking $\lambda = \lambda_j$ in the above equation and multiplying both sides by $a^{-n}\lambda_j^{-n/p}$ gives 
\[
a^{-n}\lambda_j^{-n/p}N_C^0(\lambda_j) \ge a^{-n} \lambda_j^{-n/p} N_U^0(a^p\lambda_j) + \sum_{i=1}^m a^{-n}\lambda_j^{-n/p} N_C^0(a_i^p\lambda_j).
\]
Letting $j \to \infty$ now yields
\begin{align*}
a^{-n} c^0 
&\ge \overline \beta(U) + a^{-n} \sum_{i=1}^m a_i^n c^0 \\
&\ge \overline \beta(U) + a^{-n}(1 - a^n - \eps)c^0.
\end{align*}
Therefore
\[
\overline\beta(U) \le (1+a^{-n}\eps)c^0 
\]
and hence $\overline\beta(U) \le c^0$ since $\eps$ was arbitrary. This proves that $\overline\beta(U) = \underline\beta(U) = c^0$ and so the Weyl law $N_U^0(\lambda) \sim c^0\lambda^{n/p}$ holds for $U$.  
\end{proof}

\section{The Neumann Monotonicity Inequality}
\label{neuin}

Assume that $V,W$ are open subsets of $U$ with $\cl U=\cl V\cup \cl W$ and that $\cl V\cap \cl W$ has measure 0.  Recall that when $p=2$, the Neumann eigenvalues of the Laplacian satisfy a domain monotonicity inequality of the form 
\[
N_U(\lambda) \le N_V(\lambda) + N_W(\lambda).
\]
We want to show that a similar inequality holds for arbitrary $p$.  Again the first step is to prove an energy inequality.

\begin{lem}[See \cite{G} Lemma 3.2.A]
\label{ein}
Assume $V,W$ are open subsets of $U$ with $\cl U= \cl V\cup \cl W$ and that $\cl V\cap \cl W$ has measure 0.  Let $u\in W^{1,p}(U)$ and then define $v = u|_V \in W^{1,p}(V)$ and $w = u|_W\in W^{1,p}(W)$ so that $u=v+w$.  Assume neither $v$ nor $w$ is identically zero. Then $E(u) \ge \min\{E(v),E(w)\}$. 
\end{lem}

\begin{proof}
Without loss of generality we can assume that $E(v) \le E(w)$.  Define the numbers 
\[
a := \int_V \vert \grad v\vert^p, \quad b := \int_W \vert \grad w\vert^p,\quad c := \int_V \vert v\vert^p,\quad d := \int_W \vert w\vert^p.
\]
Elementary manipulations show that 
\[
\frac{a+b}{c+d} \ge \frac{a}{c} \iff \frac{a}{c}\le \frac{b}{d}
\]
and the inequality on the right holds since $E(w) \ge E(v)$. 
It follows that 
\[
E(u) = \frac{a + b}{c+d} \ge \frac{a}{c} = E(v),
\]
as needed.
\end{proof}

%

The next step is a monotonicity inequality.  

\begin{prop}[See \cite{G} 3.2.A$_2$]
\label{ndm}
Assume $V,W$ are open subsets of $U$ with $\cl U = \cl V\cup \cl W$ and that $\cl V\cap \cl W$ has measure 0.  Then $N_U(\lambda) \le N_V(\lambda) + N_W(\lambda)$.
\end{prop}

\begin{proof}
Define the sets 
\begin{gather*}
A := \{v\in  \mathcal M(V):\, E(v) < \lambda\},\\
B := \{w\in  \mathcal M(W):\, E(w) < \lambda\},\\
C := \{u\in  \mathcal M(U):\, E(u) < \lambda\}.
\end{gather*}
Then $N_U(\lambda) = \ind(C)$ and $N_V(\lambda) + N_W(\lambda) = \ind(A) + \ind(B)$.  Hence it is enough to show that
\[
\ind(C) \le \ind(A) + \ind(B).
\]
To see this, define the sets 
\begin{gather*}
C_A := \{u\in C:\, u\vert_V \not\equiv 0 \text{ and } E(u\vert_V) < \lambda \},\\
C_B := \{u\in C:\, u\vert_W \not\equiv 0 \text{ and } E(u\vert_W) < \lambda \}.
\end{gather*}
Lemma \ref{ein} shows that $C \subset C_A \cup C_B$.  Moreover, the sets $C_A$ and $C_B$ are open in $C$ and hence $\ind(C) \le \ind(C_A) + \ind(C_B)$ by Proposition \ref{st}(ii).  But there is an odd continuous map
\begin{gather*}
C_A  \to A,\quad u \mapsto \frac{u\vert_V}{\|u\vert_V\|}
\end{gather*}
and therefore $\ind(C_A) \le \ind(A)$ by Proposition \ref{st}(i).  The same argument shows that  $\ind(C_B)\le \ind(B)$ and thus $\ind(C) \le \ind(A) + \ind(B)$.  The result follows.
\end{proof}

As in the Dirichlet case, this can be combined with the scaling properties of the energy to give a more general monotonicity inequality.  Suppose $U,U_1,\hdots,U_n$ are open in $\R^n$ and $a_1,\hdots,a_m$ are positive real numbers.  Then we will write 
\[
\sum_{i=1}^m a_iU_i \approx U
\]
if there exist elements $b_1,\hdots,b_n\in \R^n$ such that
\[
\cl U = \bigcup_{i=1}^m (\cl{a_iU + b_i})
\]
and each intersection $(\cl{a_iU_i + b_i})\cap (\cl{a_jU_j + b_j})$ has measure zero.  Using Proposition \ref{ndm}, the scaling properties of the energy, and induction shows that 
\begin{equation}
\label{ndm}
\sum_{i=1}^m a_iU_i \approx U \implies N_U(\lambda) \le \sum_{i=1}^m N_{U_i}(a_i^p\lambda).
\end{equation}
We will refer to this as the Neumann domain monotonicity inequality.

%

\section{The Weyl Law for Neumann Eigenvalues}

In this section we prove the Neumann Weyl law for domains in $\R^n$.  The first step is to prove the Weyl law on a cube. 
\begin{lem}
Let $C$ be the unit cube in $\R^n$ and define $g(\lambda) = \lambda^{-n/p}N_C(\lambda)$.  Then $g$ tends to a limit as $\lambda\to \infty$.  
\end{lem} 

\begin{proof}
Choose sequences $(\lambda_j)$, $(\mu_k)$ so that 
\[
\limsup_{\lambda\to \infty} g(\lambda) = \lim_{j\to\infty} g(\lambda_j), \quad \liminf_{\lambda\to\infty}g(\lambda) = \lim_{k\to\infty} g(\mu_k).
\]
Now fix some $k$ and consider $j$ large.  Let $\eps_j \ge 0$ be the smallest  number such that $(\mu_k/\lambda_j + \eps_j)^{1/p}$ is the reciprocal of an integer.  Then cubes of volume $(\mu_k/\lambda_j + \eps_j)^{n/p}$ partition the cube $C$.  Let $M_j$ be the number of cubes in such a partition.  

We now estimate $\eps_j$ and $M_j$.  Define $t_j := \mu_k/\lambda_j$ and then let $\ell_j$ be the largest integer such that 
\[
\ell_j \le t_j^{-1/p}.
\]  
Then $\eps_j = \ell_j^{-p} - t_j \ge 0$.  Moreover $t_j^{-1/p} - 1 \le \ell_j \le t_j^{-1/p}$ and thus
\[
\eps_j \lambda_j \le \left((t_j^{-1/p} - 1)^{-p} - t_j\right)\frac{\mu_k}{t_j} \to 0, \quad \text{as } j\to \infty.
\]
Also notice that $M_j = \ell_j^n$ and therefore 
\[
(t_j^{-1/p} - 1)^n \left(\frac{\mu_k}{t_j}\right)^{-n/p}  \le M_j\lambda_j^{-n/p} \le  \mu_k^{-n/p}
\]
so that $M_j\lambda_j^{-n/p} \to \mu_k^{-n/p}$ as $j\to \infty$.

Given these estimates, the result can be obtained as follows.  Observe that 
\[
\sum_{\ell=1}^{M_j} (\mu_k/\lambda_j + \eps_j)^{1/p}C \approx C
\]
and hence the domain monotonicity inequality gives
\[
N_C(\lambda) \le M_jN_C(\mu_k\lambda/\lambda_j + \eps_j \lambda).
\]
Choosing $\lambda = \lambda_j$ and multiplying both sides by $\lambda_j^{-n/p}$ yields 
\[
\lambda_j^{-n/p}N_C(\lambda_j) \le M_j\lambda_j^{-n/p}N_C(\mu_k + \eps_j\lambda_j). 
\]
Letting $j\to \infty$ and using the fact that $M_j\lambda_j^{-n/p} \to \mu_k^{-n/p}$ and $\eps_j\lambda_j \to 0$ this gives 
\[
\limsup_{\lambda\to \infty} g(\lambda) \le \mu_k^{-n/p}N_C(\mu_k).
\]
Finally let $k\to \infty$ to get 
\[
\limsup_{\lambda\to\infty} g(\lambda) \le \liminf_{\lambda\to\infty} g(\lambda),
\]
as needed.
\end{proof}

Define $c = \lim_{\lambda \to \infty} g(\lambda)$.  It is obvious that $c^0\le  c$ and hence $c > 0$.  The next proposition shows that $c < \infty$.    It will be shown in a later section that actually $c^0 =  c$.

\begin{prop}
Set $c = \lim_{\lambda\to\infty} \lambda^{-n/p}N_C(\lambda)$.  Then $c < \infty$.  
\end{prop}

\begin{proof}
As above, let $g(\lambda) = \lambda^{-n/p}N_C(\lambda)$.  Fix some $\lambda > 0$ and some positive integer $k$.  The unit cube can be partitioned into $k^n$ cubes of volume $k^{-n}$.  Hence  
\[
\sum_{\ell=1}^{k^n} k\inv C \approx C
\]
and so the domain monotonicity inequality implies 
\[
N_C(\lambda) \le k^nN_C(k^{-p}\lambda).
\]
Multiplying by $\lambda^{-n/p}$ gives
\[
g(\lambda) = \lambda^{-n/p} N_C(\lambda) \le (k^{-p}\lambda)^{-n/p} N_C(k^{-p}\lambda) = g(k^{-p}\lambda).  
\]
Now $g$ has some finite maximum $C_2$ on the interval $[1,2^p]$.  Since 
\[
(k+1)^p \le (2k)^p, \quad \text{for all } k\ge 1
\]
it follows that 
\[
[1,\infty) = \bigcup_{k=1}^\infty [k^p,k^p2^p].
\]
Therefore, given an arbitrary $\lambda \ge 1$, there exists some positive integer $k$ such that $k^{-p}\lambda \in [1,2^p]$.  Using the above inequality, it follows that 
\[
g(\lambda) \le g(k^{-p}\lambda) \le C_2.
\]
This proves that $c \le C_2 < \infty$, as needed.
\end{proof}

The proof of the following proposition is somewhat technical so we delay it until Section \ref{GB}.  

\begin{prop}
\label{gb}
Let $W\subset \R^n$ be a bounded open set with Lipschitz boundary.  Then  
\[
N_W(\lambda) \le C_2\vol(W)\lambda^{n/p},
\]
as $\lambda\to \infty$.  Here $C_2$ is a constant that depends on $n,$ $p$, and the Lipschitz constant of $W$.   
\end{prop}

%
%

We can now prove the Neumann Weyl law for a general domain in $\R^n$.  

\begin{theorem}
Let $U \subset \R^n$ be an open bounded set with Lipschitz boundary.  Then we have $N_V(\lambda) \sim  c \vol(V) \lambda^{n/p}$ as $\lambda\to \infty$.  
\end{theorem} 

\begin{proof}
Without loss of generality we can assume that $U$ has unit volume.  Define the numbers
\[
\underline \gamma(U) = \lim\inf_{\lambda\to\infty} \lambda^{-n/p}N_U(\lambda), \quad \overline \gamma(U) = \limsup_{\lambda\to \infty} \lambda^{-n/p}N_U(\lambda).  
\]
Given $\eps >0$, there exist numbers $a_1,\hdots,a_m$ and a set $V$ with  Lipschitz boundary such that $\vol(V) \le \eps$ and 
\[
V + \sum_{i=1}^m a_iC \approx U.
\]
Moreover, it is possible to choose $V$ so that $\operatorname{lip}(V) \le K$ for some constant $K$ that depends only on $\operatorname{lip}(U)$ and the dimension $n$.  In particular, this means that $N_V(\lambda)$ satisfies a growth bound 
\[
N_V(\lambda) \le C_2\vol(V) \lambda^{n/p} \le C_2\eps \lambda^{n/p}
\]
as $\lambda\to \infty$ where the constant $C_2$ is independent of $\eps$.
  
Applying the domain monotonicity inequality shows that 
\[
N_U(\lambda) \le N_V(\lambda) + \sum_{i=1}^m N_C(a_i^p\lambda). 
\]
But $N_C(a_i^p\lambda) \sim  c a_i^n \lambda^{n/p}$ as $\lambda\to\infty$ by the Weyl law for cubes. Hence multiplying both sides of the above inequality by $\lambda^{-n/p}$ and letting $\lambda\to \infty$ it follows that 
\[
\overline \gamma(U) \le   c \left(\sum_{i=1}^m a_i^n\right) + C_2\eps =  c (1-\eps) + C_2\eps.
\]
Since $\eps$ was arbitrary, this implies $\overline \gamma(U) \le  c$.  

It remains to show that $\underline \gamma(U) \ge  c$.  To this end, fix $a > 0$ so that some translate of $aU$ is contained in $C$.  Let $\eps > 0$ and choose numbers $a_1,\hdots,a_m$ and a set $V$ with Lipschitz boundary such that $\vol(V)\le \eps$ and 
\[
C \approx V + aU + \sum_{i=1}^m a_i C.
\]
It is possible to pick $V$ so that $\operatorname{lip}(V) \le K$ where $K$ is some constant depending only on $\operatorname{lip}(U)$ and $n$.  
Notice then that $\sum_{i=1}^m a_i^n = 1 - a^n - \vol(V)$. 
The domain monotonicity inequality gives  
\[
N_C(\lambda) \le N_V(\lambda) + N_U(a^p\lambda) + \sum_{i=1}^m N_C(a_i^p\lambda).
\]
Pick a sequence $\lambda_j\to\infty$ so that 
\[
a^{-n}\lambda_j^{-n/p} N_U(a^p\lambda_j) \to \underline\gamma(U).
\]
Taking $\lambda = \lambda_j$ in the above equation and multiplying both sides by $a^{-n}\lambda_j^{-n/p}$ yields 
\[
a^{-n}\lambda_j^{-n/p} N_C(\lambda_j) \le a^{-n}\lambda_j^{-n/p}\left(N_V(\lambda_j) + N_U(a^p\lambda_j) + \sum_{i=1}^m  N_C(a_i^p \lambda_j)\right)
\]
Letting $j\to\infty$ and using the fact that $\lambda_j^{-n/p} N_V(\lambda_j) \le C_2\vol(V) \le C_2\eps$ for all large $j$, it follows that that 
\[
a^{-n}  c \le a^{-n} C_2\eps + \underline\gamma(U) + a^{-n}\sum_{i=1}^m a_i^n  c.
\]
Thus 
\begin{align*}
\underline \gamma(U) &\ge a^{-n}\left( c - \sum_{i=1}^m a_i^n  c - C_2\eps\right)\\
&= a^{-n} c - a^{-n}(1 - a^n - \vol(V)) c - a^n C_2\eps \ge  c - a^nC_2\eps. \phantom{\bigg(}
\end{align*}
But $\eps$ is arbitrary and $C_2$ is independent of $\eps$ and hence $\underline \gamma(U)\ge  c$.  This proves that $\overline \gamma(U) = \underline \gamma(U) =  c$ and so the Weyl law $N_U(\lambda) \sim  c \lambda^{n/p}$ holds for $U$.
\end{proof}

\section{Equality of Dirichlet and Neumann Constants}

In this section we prove that the constant $c^0$ in the Dirichlet Weyl law is equal to the constant $c$ in the Neumann Weyl law. 
The following lemma of Gromov (\cite{G} Lemma 3.2.E$_1$) is the key ingredient in the proof.

\begin{lem}
\label{key}
Let $U\subset \R^n$ with smooth boundary.  Let $\eps > 0$ and define the set 
\[
U_\eps := \{x\in U:\, 0 < \text{dist}(x,\bd U) < \eps\}.
\]
Let $\lambda',\lambda'' > 0$ and set 
\[
\lambda = \frac{\lambda'\lambda''}{\lambda' + \lambda'' + \eps\inv}.
\]
Then there is an inequality 
\[
N_U(\lambda^p) \le N^{0}_U((\lambda')^p) + N_{U_\eps}((\lambda'')^p).
\]
\end{lem}

\begin{proof} Let $\eps$, $\lambda$, $\lambda'$, $\lambda''$ be as in the statement of the lemma.  Define the sets
\begin{gather*}
A := \{v \in \mathcal M^0(U):\, E(v) < (\lambda')^p\},\\
B := \{w\in  \mathcal M(U_\eps):\, E(w) < (\lambda'')^p\},\\
C := \{u\in  \mathcal M(U):\, E(u) < \lambda^p\}.
\end{gather*} 
Define $\varphi:U\to \R$ by 
\[
\varphi(x) = \begin{cases}
\eps\inv \text{dist}(x,\bd U), &\text{if } x\in U_{\eps}\\
1, &\text{otherwise}
\end{cases}
\]
and then set 
\[
C_A := \{u\in C:\, E(\varphi u) < (\lambda')^p\}. 
\]
Also define  
\[
C_B := \{u\in C:\, u\vert_{U_{\eps}}\not\equiv 0 \text{ and } E(u\vert_{U_\eps}) < (\lambda'')^p\}.  
\]
We claim that $C \subset C_A \cup C_B$.  Given this the result follows.  Indeed $C_A$ and $C_B$ are open in $C$ and so Proposition \ref{st}(ii) gives 
\[
\ind(C) \le \ind(C_A) + \ind(C_B).
\]
Moreover, there are odd continuous  maps 
\begin{gather*}
C_A\to A, \quad u \mapsto \frac{\varphi u}{\|\varphi u\|} \\
C_B\to B, \quad u\mapsto \frac{u\vert_{U_\eps}}{\|u\vert_{U_\eps}\|}.
\end{gather*}
Hence Proposition \ref{st}(i) gives
\[
\ind(C_A)\le \ind(A), \quad \ind(C_B) \le \ind(B)
\]
and it follows that $\ind(C) \le \ind(A) + \ind(B)$, as needed.

It remains to show the claim.   So suppose that $u\in C$.  If $u\vert_{U_\eps} = 0$, then $\varphi u = u$ and so $E(\varphi u) = E(u) < \lambda^p < (\lambda')^p$.  Thus $u\in C_A$.  Hence we may assume that $u\vert_{U_\eps} \neq 0$.  Suppose for contradiction that $u\notin C_A\cup C_B$.    Then 
\begin{gather*}
\| \grad(\varphi u)\|_{L^p(U)} \ge \lambda' \|\varphi u\|_{L^p(U)},\\
\|\grad u\|_{L^p(U)} \ge \| \grad u\|_{L^p(U_\eps)} \ge \lambda'' \|u\|_{L^p(U_\eps)}.
\end{gather*}
Moreover, there are inequalities
\begin{gather*}
\| \grad(\varphi u)\|_{L^p(U)} \le \|\grad u\|_{L^p(U)} + \eps\inv \|u\|_{L^p(U_\eps)},\\
\|u\|_{L^p(U)} \le \|\varphi u\|_{L^p(U)} + \|u\|_{L^p(U_\eps)}.
\end{gather*}
Therefore  
\begin{align*}
\|\grad u\|_{L^p(U)} &\ge \lambda'' \|u\|_{L^p(U_\eps)} \\
&= \lambda''\left(\|u\|_{L^p(U_\eps)} + \|\varphi u\|_{L^p(U)}\right) - \lambda'' \|\varphi u\|_{L^p(U)} \phantom{\frac{\lambda'}{\lambda''}}\\
&\ge \lambda'' \|u\|_{L^p(U)} - \frac{\lambda''}{\lambda'} \|\grad(\varphi u)\|_{L^p(U)} \\
&\ge \lambda'' \|u\|_{L^p(U)} - \frac{\lambda''}{\lambda'}\left(\|\grad u\|_{L^p(U)} + \eps\inv \|u\|_{L^p(U_\eps)}\right)\\
&\ge \lambda'' \|u\|_{L^p(U)} - \left(\frac{\lambda''}{\lambda'} + \frac{1}{\eps\lambda'}\right)\|\grad u\|_{L^p(U)}.
\end{align*}
Rearranging and using the definition of $\lambda$ gives 
\[
\|\grad u\|_{L^p(U)} \ge \lambda \|u\|_{L^p(U)}.
\]
This is a contradiction, and the claim follows.
\end{proof}

\begin{theorem}
Let $c^0$ be the constant appearing in the Dirichlet Weyl law and let $c$ be the constant appearing the Neumann Weyl law.  Then $c^0 =  c$.  
\end{theorem}

\begin{proof}
Let $B$ be the unit ball in $\R^n$.  Fix some small $\eta > 0$.  Let $\eps > 0$ and for each $\lambda > 0$ put 
\[
\lambda' = (1+\eta)\lambda, \quad \lambda'' = \frac{\lambda+\eps\inv}{\eta}
\]
so that 
\[
\lambda = \frac{\lambda'\lambda''}{\lambda' + \lambda'' + \eps\inv}.
\]
Notice that for $\lambda$ large enough there is an inequality 
$
\lambda'' \le {2\lambda}/{\eta}.
$
Using this inequality in conjunction with Lemma \ref{key} shows that 
\begin{align*}
\lambda^{-n}N_B(\lambda^p) &\le \lambda^{-n}N^{0}_B((1+\eta)^p\lambda^p) + \lambda^{-n}N_{B_\eps}(2^p\lambda^p/\eta^p)\\
&\le \lambda^{-n}N^{0}_B((1+\eta)^p\lambda^p) + C_2 \vol(B_\eps)(2/\eta)^{n}
\end{align*}
for all large $\lambda$.  Letting $\lambda\to \infty$, this implies that 
\[
 c \le c^0(1+\eta)^{n} + C_2 \vol(B_\eps)(2/\eta)^{n}.
\]
Taking $\eps \to 0$ this gives $ c\le c^0(1+\eta)^{n}$, and then letting $\eta\to 0$ yields 
$
 c\le c^0.
$
The opposite inequality $c^0 \le c$ is clear, and the result follows.
\end{proof}

\section{The Weyl Law on Closed Manifolds}

Let $(X^n,g)$ be a closed Riemannian manifold.  One defines the $p$-Laplacian $\lap_p u = \div(\vert \grad u\vert^{p-2} \grad u)$ and the space 
\[
\mathcal M(X) = \left\{u\in W^{1,p}(X):\, \|u\|_{W^{1,p}} = 1\right\}.
\]
The variational spectrum of $\lap_p$ on $X$ and the counting function $N_X(\lambda)$ are then defined as before via a min-max procedure involving the cohomological index.  The goal of this section is to show that a Weyl law $N_X(\lambda) \sim c \vol(X) \lambda^{n/p}$ holds on $X$, thus proving Theorem A.

As a first step consider some $U\subset X$ with Lipschitz boundary.  Suppose $U$ is small enough that we can find a chart $\varphi:V\to U$ where the metric $g = (g_{ij})$ on $U$ satisfies 
\[
(1-\eps)^2(\delta_{ij}) \le (g_{ij}) \le (1+\eps)^2(\delta_{ij}).
\]
We will call such a set $\eps$-admissible.  
If $U$ is $\eps$-admissible then energy function on $U$ is given by 
\[
E_U(u) = \frac{\int_V \left\vert g^{ij} \frac{\bd u}{\bd x_i}\frac{\bd u}{\bd x_j}\right\vert^{p/2} \sqrt g \, dx}{\int_V \vert u\vert^p \sqrt g\, dx}
\]
and hence there are inequalities 
\[
K_1(\eps) E_V(u) \le E_U(u) \le K_2(\eps) E_V(u)
\]
where 
\[
K_1(\eps),K_2(\eps) \to 1, \quad \text{as } \eps \to 0.
\]
It follows that there are comparisons 
\[
N^0_V(K_2(\eps)\inv\lambda) \le N^0_U(\lambda) \le N_U(\lambda) \le N_V(K_1(\eps)\inv\lambda)
\]
for every $\lambda > 0$.  There is also a volume comparison 
\[
(1-\eps)^n\vol(V) \le \vol(U) \le (1+\eps)^n\vol(V).
\]
It is now possible to prove a Weyl law for $\lap_p$ on $X$.

\begin{customthm}{A}
Let $(X^n,g)$ be a closed Riemannian manifold.  Then there is a Weyl law $N_X(\lambda) \sim c\vol(X)\lambda^{n/p}$.
\end{customthm}

\begin{proof}
Without loss of generality we may assume that $\vol(X) = 1$.  Define the quantities 
\[
\underline\gamma(X) = \liminf_{\lambda\to\infty} \lambda^{-n/p}N_X(\lambda), \quad \overline\gamma(X) = \limsup_{\lambda\to\infty} \lambda^{-n/p}N_X(\lambda).
\]
Pick some $\eps > 0$. Choose disjoint $\eps$-admissible open sets $U_1,\hdots,U_m$ in $M$ such that 
\[
\sum_{i=1}^m \vol(U_i) \ge 1 - \eps.
\]
For each $i$, let $\varphi_i:V_i\to U_i$ be a chart where the metric is almost Euclidean.  
Arguing as in Section \ref{dirin} there is an inequality  
\[
N_M(\lambda) \ge \sum_{i=1}^m N_{U_i}^0(\lambda) \ge \sum_{i=1}^m N^0_{V_i}(K_2(\eps)\inv\lambda).
\]
Multiplying by $\lambda^{-n/p}$ and then letting $\lambda\to \infty$ and using the Weyl law for domains in Euclidean space, this implies 
\[
\underline \gamma(M) \ge \sum_{i=1}^m c \vol(V_i) K_2(\eps)^{-n/p} \ge c (1+\eps)^{-n}(1-\eps)K_2(\eps)^{-n/p}.
\]
Letting $\eps \to 0$ it follows that $\underline \gamma(M) \ge c$.  

It remains to show that $\overline \gamma(M) \le c$.  To this end, let $\eps > 0$ and then choose $\eps$-admissible sets $U_1,\hdots,U_m$ so that 
\[
M = \cl U_1 \cup \hdots \cup \cl U_m,
\]
and each intersection $\cl U_i \cap \cl U_j$ has measure 0.  Such sets can be constructed using the argument in the proof of 4.2 in \cite{LMN}.   For each $i$, let $\varphi_i:V_i\to U_i$ be a chart where the metric is almost Euclidean.  Arguing as in Section \ref{neuin} there is an inequality
\[
N_M(\lambda) \le \sum_{i=1}^m N_{U_i}(\lambda) \le \sum_{i=1}^m N_{V_i}(K_1(\eps)\inv\lambda).
\]
Multiplying by $\lambda^{-n/p}$ and letting $\lambda\to \infty$ and using the Weyl law for domains in Euclidean space, this implies 
\begin{align*}
\overline\gamma(M) &\le \sum_{i=1}^m c \vol(V_i) K_1(\eps)^{-n/p} \le c(1-\eps)^{-n}K_1(\eps)^{-n/p}.
\end{align*}
Letting $\eps\to 0$ it follows that $\overline \gamma(M) \le c$.  This proves the result.
\end{proof} 

\section{Growth Bound for the Neumann Counting Function} \label{GB}

The goal of this final section is to prove the growth bound in Proposition \ref{gb}.  The argument uses Sobolev extension operators.  

\begin{defn}
Let $U\subset \R^n$ be a bounded open set.  Then $U$ is $L$-Lipschitz provided there is a covering of $\bd U$ by balls $(B_i)$ with the following property: for each $i$ there is an $L$-Lipschitz function 
$
f_i:\R^{n-1}\to \R
$
such that, up to translation and rotation, $U\cap B_i$ coincides with the set 
\[\{(x_1,\hdots,x_n) \in \R^n:\, x_n < f(x_1,\hdots,x_{n-1})\} \cap B_i.
\]
\end{defn}

Jones introduced the following more general class of domains in \cite{J}.  

\begin{defn}
Let $U$ be an open set in $\R^n$ and let $\eps,\delta > 0$.  Then $U$ is called an $(\eps,\delta)$-domain provided for every $x,y\in U$ with $\vert x-y\vert < \delta$ there is a rectifiable arc $\gamma$ joining $x$ to $y$ with 
\begin{enumerate}
\item[(i)] $\text{Length}(\gamma) \le \dfrac{\vert x-y\vert}{\eps}$,\\
\item[(ii)] $\operatorname{dist}(z,\bd U) \ge \dfrac{\eps \vert x-z\vert \vert y-z\vert}{\vert x-y\vert}$ for all $z\in \gamma$. 
\end{enumerate}
\end{defn}

Every $L$-Lipschitz domain $U$ is an $(\eps,\delta)$ domain for some choice of $\eps$ and $\delta$.  Moreover, $\eps$ can be taken to depend only on $L$ and $n$.  
The following result is due to Jones \cite{J}.   

\begin{theorem}
\label{ext}
Let $U\subset \R^n$ be an $(\eps,\delta)$-domain.  Then there exists a continuous linear map 
\[
\mathcal E: W^{1,p}(U) \to W^{1,p}(\R^n)
\]
with the property that $(\mathcal Eu)\vert_U = u$.  Moreover, the norm of $\mathcal E$ depends only on $n$, $p$, $\eps$, and $\delta$.  
\end{theorem}

It is now possible to prove the growth bound.  

\begin{proof} (Proposition \ref{gb})
Let $U\subset \R^n$ be a bounded open set and assume that $U$ is $L$-Lipschitz.  We need to check that 
\[
N_U(\lambda) \le C \vol(U) \lambda^{n/p}, \quad\text{as } \lambda\to \infty
\]
for some constant $C$ that depends only on $n$, $p$, and $L$.  Notice that if $U$ is $L$-Lipschitz then so is the scaled copy $aU$ for any $a > 0$.  Since the above inequality is scale invariant, it suffices to show that 
\[
N_{aU}(\lambda) \le C \vol(aU) \lambda^{n/p}, \quad \text{as } \lambda\to \infty
\]
for some $a > 0$.  

Since $U$ is $L$-Lipschitz, it is an $(\eps,\delta)$-domain for some choice of $\eps$ and $\delta$.  Hence for $a$ large enough, $aU$ will be an $(\eps,1)$-domain.  Choose an open set $V$ containing the closure of $U$ with $\vol(V) \le 2\vol(U)$.  Then it is still true that $\vol(aV) \le 2\vol(aU)$.  Moreover, for $a$ large enough $aV$ will contain the 1-neighborhood of $aU$.  
For notational convenience put 
\[
\widetilde U = aU\quad \text{and}\quad \widetilde V = aV.
\]
By Theorem \ref{ext} there is an extension operator 
\[
\mathcal E:W^{1,p}(\widetilde U) \to W^{1,p}(\R^n).  
\]
Since $\delta = 1$ and $\eps$ depends only on $L$ and $n$, it follows that the norm of $\mathcal E$ depends only on $n$, $p$, and $L$.  

Let $\zeta$ be a cutoff function with $\zeta\equiv 1$ on $\widetilde U$ and $\zeta\equiv 0$ outside of $\widetilde V$.  It is possible to choose $\zeta$ so that $\vert \grad \zeta\vert \le 2$ everywhere.   Define an odd continuous mapping
\begin{gather*}
G: W^{1,p}(\widetilde U) \to W^{1,p}_0(\widetilde V),\\
u \mapsto \zeta\cdot \mathcal Eu.
\end{gather*}
There is an estimate
\begin{align*}
E(Gu) = \frac{\int_{\widetilde V} \vert \grad(\zeta\cdot \mathcal Eu)\vert^p}{\int_{\widetilde V} \vert \zeta\cdot \mathcal Eu\vert^p} &\le C \left(\frac{\int_{\widetilde V}  \vert \grad(\mathcal Eu)\vert^p  \vert \zeta\vert^p + \vert \mathcal Eu\vert^p \vert \grad \zeta\vert^p}{\int_{\widetilde U} \vert u\vert^p}\right)\\
&\le {C} \left(\frac{\int_{\widetilde V} \vert \grad(\mathcal Eu)\vert^p + \vert \mathcal E u\vert^p}{\int_{\widetilde U} \vert u\vert^p}\right)\\
&\le {C} \left(\frac{\int_{\widetilde U} \vert \grad u\vert^p + \vert u\vert^p}{\int_{\widetilde U} \vert u\vert^p} \right)\\
&\le {C} (E(u) + 1).\phantom{\bigg(}
\end{align*}
Now let $A = \{u\in W^{1,p}(\widetilde U):\, E(u) < \lambda\}$.  Then the above estimate implies that 
\[
E(Gu) \le {C}(\lambda + 1)
\]
for every $u\in A$.  By Proposition \ref{st}(i) this implies 
\[
N_{\widetilde U}(\lambda) \le N^0_{\widetilde V}\left({C}(\lambda + 1)\right) \le C \vol(\widetilde V) (\lambda + 1)^{n/p}.
\]
Therefore 
\[
N_{\widetilde U}(\lambda) \le C \vol(\widetilde U) \lambda^{n/p}
\]
for all sufficiently large $\lambda$, as needed.
\end{proof}

\bibliographystyle{plain}
\bibliography{p-laplacian.bib}

\end{document}